%% file: ms.tex
\documentclass[letterpaper, 10 pt, conference]{ieeeconf}  

\IEEEoverridecommandlockouts                              

\usepackage{times,graphicx,algorithm,algorithmicx,algpseudocode,fullpage,psfrag,amsthm,amsmath,amsfonts,verbatim,mathtools,enumerate,bm,hyperref,dsfont,tikz}
\allowdisplaybreaks

\input defs.tex

\title{\LARGE \bf Stochastic Bregman Parallel Direction Method of Multipliers for Distributed Optimization 
}

\author{Yue~Yu and Beh\c{c}et~A\c{c}\i kme\c{s}e
\thanks{
The authors are with the Department of Aeronautics and Astronautics, University of Washington, Seattle,
        WA, 98195; emails: 
        {\tt\small \{yueyu,behcet\}@uw.edu}}%
}

\begin{document}

\maketitle
\thispagestyle{empty}
\pagestyle{empty}

\begin{abstract}
Bregman parallel direction method of multipliers (BPDMM) efficiently solves distributed optimization over a network, which arises in a wide spectrum of collaborative multi-agent learning applications. In this paper, we generalize BPDMM to stochastic BPDMM, where each iteration only solves local optimization on a randomly selected subset of nodes rather than all the nodes in the network. Such generalization reduce the need for computational resources and allows applications to larger scale networks. We establish both the global convergence and the \(O(1/T)\) iteration complexity of stochastic BPDMM. We demonstrate our results via numerical examples. 
\end{abstract}

\input{introduction}

\input{preliminaries}
\input{method}

\input{convergence}
\input{experiments}

\input{conclusion}




\bibliographystyle{IEEEtran}
\bibliography{IEEEabrv,reference}

\newpage
\section*{APPENDIX}
\input{appendix.tex}

\end{document}

%% file: defs.tex
\newcommand{\argmin}{\mathop{\rm argmin}}

\newcommand{\norm}[1]{\left\lVert#1\right\rVert}
\newcommand{\mnorm}[1]{{\left\vert\kern-0.25ex\left\vert\kern-0.25ex\left\vert #1 
    \right\vert\kern-0.25ex\right\vert\kern-0.25ex\right\vert}}

\newtheorem{theorem}{Theorem}
\newtheorem{lemma}{Lemma}
\newtheorem{corollary}{Corollary}

\newtheorem{assumption}{Assumption}

\algnewcommand\algorithmicforeach{\textbf{for each}}
\algdef{S}[FOR]{ForEach}[1]{\algorithmicforeach\ #1\ \algorithmicdo}

\newcommand{\eg}{{\it e.g.}}
\newcommand{\ie}{{\it i.e.}}

%% file: introduction.tex
\section{Introduction}
Distributed optimization over a connected undirected network \(\mathcal{G}=(\mathcal{V}, \mathcal{E})\) is defined as follows
\begin{equation}
    \begin{array}{ll}
    \underset{x\in\mathcal{X}^{|\mathcal{V}|}}{\mbox{minimize}} & \sum\limits_{i\in\mathcal{V}} f_i(x_i)\\
        \mbox{subject to} & x_i=x_j, \enskip\forall \{i, j\}\in\mathcal{E} 
    \end{array}
    \label{opt: dist opt}
\end{equation}
where \(\mathcal{X}\subset\mathbb{R}^n\) is a closed convex set, \(\mathcal{X}^{\mathcal{V}}\) is the Cartesian product of \(|\mathcal{V}|\) copies of \(\mathcal{X}\), each \(f_i\) is a convex function accessible by node \(i\) only. The global optimality is achieved by local optimization on each node and efficient communication between neighboring nodes. In addition to classical applications such as formation control \cite{mesbahi2010graph}, distributed tracking \cite{li2002detection} and estimation \cite{accikmecse2014decentralized,lesser2012distributed}, problem \eqref{opt: dist opt} also arises in collaborative learning scenarios \cite{gholami2016decentralized,yahya2017collective}, where problem \eqref{opt: dist opt} represents distributed learning from data collected by multiple agents.  

There has been an increasing interest in applying multiplier methods to solve problem \eqref{opt: dist opt} \cite{wei2012distributed,meng2015proximal,deng2017parallel}. At each iteration of such methods, every primal variable is updated by optimizing a quadratic augmented Lagrangian; every dual variable is updated by numerically integrating local disagreement. Recently, Bregman parallel direction method of multipliers (PDMM) generalized the quadratic augmentation in local optimization to Bregman augmentation, which better exploits the structure of constraint set \(\mathcal{X}\), and hence leads to significant improvement in convergence speed \cite{wang2014bregman,yu2018bregman}.  

One challenge in implementing multiplier methods for problem \eqref{opt: dist opt} is that a local optimization problem needs to be solved on every node in parallel at each iteration, which requires demanding computational resources when applied to large scale networks. A popular approach to address this challenge is stochastic multiplier methods \cite{wei20131,wang2014parallel,zhu2016stochastic}, which combine multiplier methods with the idea of stochastic block coordinate descent \cite{nesterov2012efficiency,richtarik2014iteration}. At each iteration, stochastic multiplier methods only solve local optimization problems on, rather than all the nodes, a randomly selected subset of nodes. Such algorithms guarantee global convergence to optimum in expectation via proper choice of algorithm parameters. However, to our best knowledge, all existing stochastic multiplier methods use quadratic augmentation. In other words, there is no stochastic extension to Bregman augmentation based multiplier methods.

In this paper, we close this gap in the literature by proposing stochastic BPDMM, which combines the benefits of BPDMM and stochastic multiplier methods. Compared with BPDMM \cite{yu2018bregman}, it only requires solving local optimization on a randomly selected subset of nodes, which allows application to larger scale networks; compared with existing stochastic multiplier methods \cite{wei20131,wang2014parallel,zhu2016stochastic}, it extends quadratic augmented Lagrangian to Bregman augmented Lagrangian, which improves the convergence speed by better exploiting constraints structure. We establish the global convergence and \(O(1/T)\) iteration complexity of stochastic BPDMM, and demonstrate its effectiveness and efficiency via numerical examples.  

The rest of the paper is organized as follows. Section~\ref{section: preliminaries} covers necessary background and reformulates problem \eqref{opt: dist opt} with consensus constraints. Section~\ref{section: methods} develops the stochastic BPDMM, whose convergence proof is established in Section~\ref{section: convergence}. Section~\ref{section: experiments} presents numerical examples and demonstrates the advantages of stochastic BPDMM over prior work. Section~\ref{section: conclusion} concludes and comments on future directions.

%% file: preliminaries.tex
\section{Preliminaries and Background}\label{section: preliminaries}

\subsection{Notation}
Let \(\mathbb{R}\) (\(\mathbb{R}_+\)) denote the set of (nonnegative) real numbers, \(\mathbb{R}^n\) (\(\mathbb{R}^n_+\)) the set of \(n\)-dimensional (elementwise nonnegative) vectors.  Let \(\geq(\leq)\) denote elementwise inequality when applied to vectors and matrices. Let \(\langle\cdot, \cdot\rangle\) denote the dot product.  Let \(I_n\in\mathbb{R}^{n\times n}\) denote the \(n\)-dimensional identity matrix, \(\mathbf{1}_n\in\mathbb{R}^n\) the \(n\)-dimensional vector of all \(1\)s. Given matrix \(A\in\mathbb{R}^{n\times n}\), let \(A_{ij}\) denote its \((i, j)\) entry; \(A^\top\) denotes its transpose. Let \(\otimes\) denote the Kronecker product. 
\subsection{Subgradients}
Let \(f:\mathbb{R}^n\to\mathbb{R}\) be a convex function. Then \(g\in\mathbb{R}^n\) is a subgradient of \(f\) at \(u\in\mathbb{R}^n\) if and only if for any \(v\in\mathbb{R}^n\) one has
\begin{equation}
f(v)-f(u)\geq \left\langle g, v-u\right\rangle.
\label{definition of subgradient}
\end{equation}
We denote \(\partial f(u)\) the set of subgradients of \(f\) at \(u\). An important case of subdifferential is  the case of indicator function of a non-empty convex set \(\mathcal{X}\) defined as \(\delta_\mathcal{X}(x)=0\) if \(x\in\mathcal{X}\) and \(\infty\) otherwise. We will use the following results.
\begin{lemma}\cite[Theorem 27.4]{rockafellar2015convex} Given a closed convex set \(\mathcal{X}\subseteq\mathbb{R}^n\) and closed, convex, proper function \(f:\mathbb{R}^n\to \mathbb{R}\), then \(u^\star=\argmin_{u\in\mathcal{X}}\, f(u)\) if and only if \(0\in\partial (f+\delta_ \mathcal{X})(u^\star)\).
\end{lemma} 
\subsection{Mirror maps and Bregman divergence} 
Let \(\mathcal{D}\subseteq\mathbb{R}^n\) be a convex open set. We say that \(\phi:\mathcal{D}\to\mathbb{R}\) is a {\em mirror map} \cite[p.298]{bubeck2015convex} if it satisfies: 1) \(\phi\) is differentiable and strictly convex, 2) \(\nabla\phi\) takes all possible values, and 3) \(\nabla\phi\) diverges on the boundary of the closure of \(\mathcal{D}\), \ie , \(\lim_{u\to\partial \bar{\mathcal{D}}}\norm{\nabla\phi(u)}=\infty\), where \(\norm{\cdot}\) is an arbitrary norm on \(\mathbb{R}^n\). The Bregman divergence  \(B_\phi:\mathcal{D}\times\mathcal{D}\to\mathbb{R}_+\) is defined as \cite[Sec. 2.1]{censor1997parallel}
\begin{equation}
B_\phi(u, v)=\phi(u)-\phi(v)-\left\langle \nabla\phi(v), u-v\right\rangle.
\label{definition of Bregman divergence}
\end{equation}
Note that \(B_\phi(u, v)\geq 0\) and \(B_\phi(u, v)=0\) only if \(u=v\). $B_\phi$ also satisfy the following three-point identity,
\begin{equation}
\begin{aligned}
&\langle\nabla\phi(u)-\nabla\phi(v), w-u\rangle\\
=&B_\phi(w, v)-B_\phi(w, u)-B_\phi(u, v).
\end{aligned}\label{3-point property}
\end{equation}  

\subsection{Graphs and distibuted optimization}
An undirected connected graph \(\mathcal{G}=(\mathcal{V}, \mathcal{E})\) contains a vertex set \(\mathcal{V}=\{1, 2, \ldots, m\}\) and an edge set \(\mathcal{E}\subseteq \mathcal{V}\times \mathcal{V}\) such that \((i, j)\in\mathcal{E}\) if and only if \((j, i)\in\mathcal{E}\) for all \(i, j\in\mathcal{V}\). Denote \(\mathcal{N}(i)\) the set of neighbors of node \(i\) such that \(j\in\mathcal{N}(i)\) if \((i, j)\in\mathcal{E}\). 

Consider a symmetric stochastic matrix \(P\in\mathbb{R}^{|\mathcal{V}|\times |\mathcal{V}|}\) defined on the graph \(\mathcal{G}\) such that \(P_{ij}>0\) implies that \(j\in\mathcal{N}(i)\). Such a matrix \(P\) can be constructed, for example, by the graph Laplacian \cite[Proposition 3.18]{mesbahi2010graph}. If \(P\) is irreducible \cite[Lem. 8.4.1]{horn2012matrix}, then \(1\) is a simple eigenvalue of \(P\) with eigenvectors spanned by \(\mathbf{1}_{|\mathcal{V}|}\).


Let \(\mathcal{G}=(\mathcal{V}, \mathcal{E})\) denote the underlying graph over which problem \eqref{opt: dist opt}  is defined. A common approach  to solve problem is to create local copies of the design variable \(\{x_1, x_2, \ldots, x_{|\mathcal{V}|}\}\) and impose the consensus constraints: \(x_i=x_j\) for all \((i, j)\in \mathcal{E}\) \cite{bertsekas1989parallel, boyd2011distributed}. Many different consensus constraints have been proposed \cite{wei2012distributed, jakovetic2013distributed, iutzeler2013asynchronous, shi2014linear}. In this paper, we consider consensus constraints of the form:
\begin{equation}
(P\otimes I_n)x=x,
\label{consensus constraints}
\end{equation}
where \(x =[x^\top_1, x^\top_2, \ldots, x^\top_{|\mathcal{V}|}]^\top\), \(P\) is a symmetric, stochastic and irreducible matrix defined on \(\mathcal{G}\). We will focus on the following reformulation of problem \eqref{opt: dist opt},
\begin{equation}
\begin{array}{ll}
\underset{x\in\mathcal{X}^{|\mathcal{V}|}}{\mbox{minimize}} & \sum\limits_{i\in\mathcal{V}} f_i(x_i)\\
\mbox{subject to} & (P\otimes I_n) x= x.
\end{array}
\label{opt: consensus optimization problem}
\end{equation}

%% file: method.tex
\section{Stochastic Bregman Parallel Direction Method of Multipliers} \label{section: methods}

In this section, we first review BPDMM in Algorithm~\ref{alg: BPDMM}, then combine it with the stochastic node update in \cite{wang2014parallel} and propose sBPDMM in Algorithm~\ref{alg: sBPDMM}.

BPDMM \cite{yu2018bregman} solves problem \eqref{opt: consensus optimization problem} with Algorithm~\ref{alg: BPDMM}, which combines the idea of PDMM \cite{meng2015proximal} and Bregman augmented Lagrangian \cite{wang2014bregman}. Each iteration of the algorithm include the following steps:
\begin{enumerate}[(a)]
    \item \emph{Mirror averaging} Step \eqref{BPDMM: mirror averaging} computes a nodal mirror average of neighboring nodes' variables, and can be further decomposed as follows:
    \begin{subequations}
    \begin{align}
    \nabla\Phi(z^t)=&(P\otimes I_n) \nabla \Phi (x^t)\label{eqn: mirror avg 1}\\
    y^t=&\underset{y\in\mathcal{X}^{|\mathcal{V}|}}{\argmin} \enskip B_\Phi(y, z^t)\label{eqn: mirror avg 2}
    \end{align}
    \label{eqn: mirror decomp}
    \end{subequations}
    where \(\Phi(x)=\sum_{i\in\mathcal{V}}\phi(x_i)\).
    Therefore this step is equivalent to first apply \(\nabla\Phi\) to \(x^t\), then run an average step, followed by \((\nabla\Phi)^{-1}\), and finally a projection step. See Fig.~\ref{fig: mirror averaging} for an illustration. 
    \item \emph{Local optimization} Step \eqref{BPDMM: primal subproblem} optimizes a nodal augmented Lagrangian. In particular, the Bregman divergence term in the objective of \eqref{BPDMM: primal subproblem} augments the nodal Lagrangian by penalizing the difference from the nodal mirror average.  
    \item \emph{Disagreement integration} Step \eqref{BPDMM: dual update} is a discrete integration of the disagreement between neighboring nodes. Such integration is equivalent to a spring dynamics among neighboring nodes and improves the disturbance rejection performance of the algorithm. See \cite{wang2010control,yu2018mass} for a detailed discussion. 
\end{enumerate}
\input{mirror.tex}

Both mirror averaging step \eqref{BPDMM: mirror averaging} and disagreement integration step \eqref{BPDMM: dual update} have close-form update when the constraint set \(\mathcal{X}\) is structured, \eg , \(\mathcal{X}\) is \(\mathbb{R}^n\) or the probability simplex \cite{yu2018bregman}. On the other hand, the local optimization step \eqref{BPDMM: primal subproblem} typically requires an iterative algorithm itself, \eg , mirror descent method \cite{nemirovsky1983problem}. Hence the main computational effort of implementing Algorithm~\ref{alg: BPDMM} is caused by the local optimization step \eqref{BPDMM: primal subproblem}. At each iteration, Algorithm~\ref{alg: BPDMM} requires at least \(|\mathcal{V}|\) processors, one assigned to each node, to solve optimization \eqref{BPDMM: primal subproblem} in parallel. Such requirements are computationally demanding for large scale networks. 

\begin{algorithm}
\caption{BPDMM}
\begin{algorithmic}[h]
\Require Parameters: \(\tau, \rho >0\); initial point \(x^{0}\in(\mathcal{X}\cap\mathcal{D})^{|\mathcal{V}|} , \mu^{0}\in\mathbb{R}^{|\mathcal{V}|n}\).
\ForAll{\(t=0, 1, 2, \ldots\)}
\State{\begin{subequations}
\begin{align}
&y_i^{t}  =  \underset{y_i\in \mathcal{X}}{\argmin}\sum_{j\in\mathcal{N}(i)} P_{ij} B_\phi(y_i, x_j^{t}),\enskip \forall i\in\mathcal{V}\label{BPDMM: mirror averaging}\\
&\begin{aligned}
x_{i}^{t+1} = \,  \underset{x_i\in\mathcal{X}}{\argmin}  \,\,  &f_i(x_i) +\langle x_i, \mu_i^{t} - \sum_{j\in\mathcal{N}(i)}P_{ij}\mu^{t}_j\rangle\\
 &+ \rho B_\phi(x_i, y_i^{t}),\enskip \forall i\in\mathcal{V} \label{BPDMM: primal subproblem}
\end{aligned}
\end{align}\label{BPDMM primal update}
\end{subequations}
\begin{equation}
    \mu_i^{t+1} =  \mu_i^{t}+\tau x_i^{t+1}-\tau\sum_{j\in\mathcal{N}(i)}P_{ij}x_j^{t+1},\enskip \forall i\in\mathcal{V}\label{BPDMM: dual update}
\end{equation}
}
\EndFor
\end{algorithmic}
\label{alg: BPDMM}
\end{algorithm}
 
In order to address this challenge, we propose Algorithm~\ref{alg: sBPDMM}, which uses a stochastic node update \cite{wei20131,wang2014parallel,zhu2016stochastic}. Compared with Algorithm~\ref{alg: BPDMM}, each iteration of Algorithm~\ref{alg: sBPDMM} only execute local optimization step on a set of randomly selected nodes, which requires less number of processors running in parallel. This flexibility reduce the requirements on the total computation power of the network, and allows BPDMM to be applicable much larger scale networks.

\begin{algorithm}
\caption{stochastic BPDMM}
\begin{algorithmic}[h]
\Require Parameters: \(\tau, \rho>0\); initial point \(x^{0}\in(\mathcal{X}\cap \mathcal{D})^{|\mathcal{V}|}, \mu^{0}\in\mathbb{R}^{|\mathcal{V}|n}\). 
\ForAll{\(t=0, 1, 2, \ldots\)}
\State{Randomly select a subset of nodes \(\mathcal{S}_{t+1} \subset\mathcal{V}\).
\begin{subequations}
\begin{align}
    &y_i^{t}  =  \underset{y_i\in \mathcal{X} }{\argmin}\sum_{j\in\mathcal{N}(i)} P_{ij} B_\phi(y_i, x_j^{t}), \enskip \forall i\in\mathcal{S}_{t+1}\label{mirror averaging}\\
    &\begin{aligned}
    x_{i}^{t+1} =  \underset{x_i\in\mathcal{X}}{\argmin}  \,\, & f_i(x_i) +\langle x_i, \mu_i^{t}-\sum_{j\in\mathcal{N}(i)}P_{ij}\mu^{t}_j\rangle\\
    & + \rho B_\phi(x_i, y_i^{t}),\enskip \forall i\in\mathcal{S}_{t+1}\label{primal update}
\end{aligned}\\
& x_i^{t+1}=x_i^t, \enskip \forall i\in \mathcal{\mathcal{V}}\setminus\mathcal{S}_{t+1}\label{not updated variable}
\end{align}\label{sBPDMM primal update}
\end{subequations}
\begin{equation}
 \mu_i^{t+1} =  \mu_i^{t}+\tau x_i^{t+1}-\tau\sum_{j\in\mathcal{N}(i)}P_{ij}x_j^{t+1}, \enskip \forall i\in\mathcal{V} \label{dual update} 
\end{equation}
}
\EndFor
\end{algorithmic}
\label{alg: sBPDMM}
\end{algorithm}

Although the generalization from Algorithm~\ref{alg: BPDMM} to Algorithm~\ref{alg: sBPDMM} seems straightforward, the generalization in the corresponding convergence proof requires more careful treatment. In particular, the convergence proof of Algorithm~\ref{alg: BPDMM} in \cite{yu2018bregman} hinges on a monotonically non-increasing non-negative Lyapunov function for full primal update in \eqref{BPDMM primal update} with carefully chosen algorithm parameters. In order to generalize such proof to Algoritjm~\ref{alg: sBPDMM}, we need to answer the following questions:
\begin{itemize}
    \item How to find a monotonically non-increasing non-negative Lyapunov function for stochastic partial primal update in \eqref{sBPDMM primal update}? 
    \item How does the randomly selected node set \(\mathcal{S}_{t+1}\) affect the choice of algorithm parameters?
\end{itemize}

In the sequel, we aim to answer theses questions and establish the convergence proof of Algorithm~\ref{alg: sBPDMM}.

%% file: mirror.tex
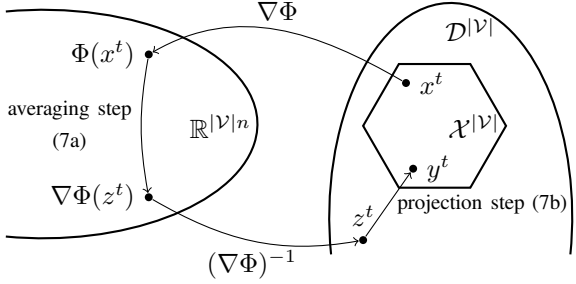
\begin{figure}
    \centering
    \begin{tikzpicture}[scale=0.95]
    \draw[thick] (-6,-1.55) arc(-100:100:3 and 1.6);
    \node[circle,fill=black,inner sep=0pt,minimum size=3pt,label=left:{\(\Phi(x^t)\)}] (a) at (-4, 1) {};
    \node[circle,fill=black,inner sep=0pt,minimum size=3pt,label=left:{\(\nabla \Phi (z^t)\)}] (b) at (-4, -1){};
    \draw[->] (a) to [out=-100,in=100,edge node={node [midway,left,align=center] {\footnotesize{averaging step}\\ \footnotesize{\eqref{eqn: mirror avg 1}}}}] (b);
    
    \node at (-3, 0) {\(\mathbb{R}^{|\mathcal{V}|n}\)};

    \draw[thick] (1.9,-1.8) arc(-10:190:1.7 and 3);
    \draw[thick] (0:1) \foreach \x in {60,120,...,359} {
                 -- (\x:1)
             }-- cycle (90:1);
    \node at (0.5, 1.35) {\(\mathcal{D}^{|\mathcal{V}|}\)};
    \node at (0.55, 0) {\(\mathcal{X}^{|\mathcal{V}|}\)};
    \node[circle,fill=black,inner sep=0pt,minimum size=3pt,label=right:{\(x^t\)}] (c) at (-0.4, 0.6){};
    
    \node[circle,fill=black,inner sep=0pt,minimum size=3pt,label=right:{\(y^t\)}] (d) at (-0.3, -0.6){};
    
    \node[circle,fill=black,inner sep=0pt,minimum size=3pt,label=above:{\(z^t\)}] (e) at (-1, -1.6){};
    
    \draw[->] (c) to [out=150, in=30, edge node={node [midway,above] {\(\nabla \Phi\)}}] (a);
    \draw[->] (b) to [out=-30, in=190, edge node={node [midway,below] {\((\nabla \Phi)^{-1}\)}}] (e);
    \draw[->] (e) to [edge node={node [midway,right,align=center] {\footnotesize{projection step \eqref{eqn: mirror avg 2}}}}] (d);
    
    \end{tikzpicture}
    \caption{Mirror averaging}
    \label{fig: mirror averaging}
\end{figure}

%% file: convergence.tex
\section{Convergence}
\label{section: convergence}
In this section, we prove the global convergence as well as the \(O(1/T)\) iteration complexity of Algorithm~\ref{alg: sBPDMM}. All detailed proof in this section can be found in the Appendix.

We first group our assumptions in Assumption~\ref{basic assumption}.
\begin{assumption}
\label{basic assumption}
\begin{enumerate}[(a)]
\item Function \(f_i:\mathbb{R}^n\to\mathbb{R}\cup \{+\infty\}\) are closed, proper and convex for all \(i\in\mathcal{V}\).
\item Set \(\mathcal{X}\subset\mathbb{R}^n\) is closed and convex. There exists a saddle point \((x^\star, \mu^\star)\) such that  \(x^\star_i\in\mathcal{X}\) and
\begin{subequations}
    \begin{align}
      \sum_{j\in\mathcal{V}}P_{ij}x_j^\star & =x_i^\star \label{KKT: primal} \\
      -\mu^\star_i+\sum_{j\in\mathcal{V}} P_{ij}\mu^\star_j & \in \partial (f_i+\delta_\mathcal{X})(x_i^\star), \label{KKT: dual}
    \end{align} \label{KKT}
\end{subequations}
for all \(i\in\mathcal{V}\).
\item  Function \(\phi:\mathcal{D}\to\mathbb{R}\) is a mirror map, where \(\mathcal{D}\) is a open convex set such that \(\mathcal{X}\) is included in its closure. In addition, function \(\phi\) is \(\alpha\)-strongly convex with respect to \(l_p\)-norm, \ie, for any \(u, v\in \mathcal{X}\), 
\begin{equation}
B_\phi(u, v)\geq \frac{\alpha}{2}\norm{u-v}_p^2.
\label{phi strong convexity}
\end{equation} 
\item Matrix \(P\) is symmetric, stochastic, irreducible and positive semi-definite. \label{P matrix assumption}
\item At each iteration \(t+1\), we assume \(|\mathcal{S}_{t+1}|/|\mathcal{V}|=\omega, 0<\omega<1.\)
\end{enumerate}
\end{assumption}

Now we start to construct the convergence proof of Algorithm~\ref{alg: sBPDMM} under Assumption~\ref{basic assumption}. The optimality condition of \eqref{primal update} is that for all \(i\in\mathcal{S}_{t+1}\), 
\begin{equation}
\begin{aligned}
&-\mu_{i}^{t}+\sum_{j\in\mathcal{V}} P_{ij}\mu_j^{t}-\rho\left(\nabla\phi(x_{i}^{t+1})-\nabla\phi(y_{i}^{t})\right)\\
\in &\,\,\partial (f_{i}+\delta_\mathcal{X})(x_{i}^{t+1})
\end{aligned}
\label{primal updata: optimality}
\end{equation}
Define the residuals of optimality conditions \eqref{primal updata: optimality} at iteration \(t\) as
\begin{equation}
\begin{aligned}
&R(t+1)\coloneqq \omega (L(x^t, \mu^\star)-L(x^\star, \mu^\star))\\
&+\rho\sum_{i\in\mathcal{S}_{t+1}}B_\phi(x_{i}^{t+1}, y_{i}^{t})+\frac{\gamma\rho}{2}\norm{((I_{|\mathcal{V}|}-P)\otimes I_n)x^{t}}_2^2,
\end{aligned}
\label{residuals}
\end{equation}
where \(\gamma>0\) and Lagrangian \(L(x, \mu)\) is defined as
\begin{equation}
    L(x, \mu)=\sum_{i\in\mathcal{V}}(f_i+\delta_\mathcal{X})(x_i)+\langle \mu, ((I_{|\mathcal{V}|}-P)\otimes I_n)x\rangle.
    \label{eqn: Lagrangian}
\end{equation}
Using \eqref{KKT} and \eqref{definition of subgradient} we can show the following
\begin{equation}
   L(x^t, \mu^\star)-L(x^\star, \mu^\star)\geq 0 
   \label{eqn: duality gap}
\end{equation}
Hence \(L(x^t, \mu^\star)-L(x^\star, \mu^\star)\) defines a \emph{running duality gap} that measures distance to optimality \cite{meng2015proximal}. Notice that given \(x^t\), \(R(t+1)\) is a random variable only depends on \(\mathcal{S}_{t+1}\) and \(\mathds{E}_{\mathcal{S}_{t+1}}\left[R(t+1)\right]=0\) implies that \(L(x^t, \mu^\star)=L(x^\star, \mu^\star)\) and \(x_i^t=x_j^t\) for all \(i, j\in\mathcal{V}\), \ie , both optimality and consensus are achieved.

In order to show \(\mathds{E}_{\mathcal{S}_{t+1}}\left[R(t+1)\right]=0\), we define the following Lyapunov function of Algorithm~\ref{alg: sBPDMM}
\begin{equation}
\begin{aligned}
V(t)\coloneqq &H(x^t, \mu^t) +\frac{\omega}{2\tau}\norm{\mu^\star-\mu^{t-1}}_2^2\\
&+\rho\sum_{i\in\mathcal{V}} B_\phi(x_i^\star, x_i^{t}).
\end{aligned}
\label{Lyapunov function}
\end{equation}
where 
\begin{equation}
    H(x^t, \mu^t)=L(x^t, \mu^t) -L(x^\star, \mu^\star)-\tau\norm{Q\otimes I_n)x^{t}}_2^2
    \label{generalized Lagrangian}
\end{equation}
with \(Q=I_{|\mathcal{V}|}-P\) and
\(\mu^{-1}\coloneqq \mu^{0}-\tau((I_{|\mathcal{V}|}-P)\otimes I_n)x^{0}\).

Compared with the one used in \cite{yu2018bregman}, the Lyapunov function \(V(t)\) defined by \eqref{Lyapunov function} contains a generalized Lagrangian \(H(x^t, \mu^t)\), which renders the positive definiteness of \(V(t)\) unclear. The following lemma shows that \(V(t)\) is indeed positive definite, and lower bounded by a Bregman divergence to the optimum.
\begin{lemma}
Suppose Assumption~\ref{basic assumption} holds, if 
\begin{equation}
\tau\leq \frac{\rho\left(\omega\alpha\sigma-\gamma\right)}{2-\omega}, \quad 0< \gamma<\omega\alpha\sigma,
\label{parameter}
\end{equation}
where \(\sigma=\min\{1, n^{\frac{2}{p}-1}\}\), \(p\) and \(\alpha\) are defined in \eqref{phi strong convexity}, then the Lyapunov function defined in \eqref{Lyapunov function} satisfy
\begin{equation}
V(t)\geq \frac{(1-\omega)\omega\alpha\sigma\rho+\gamma\rho}{(2-\omega)\omega\alpha\sigma}\sum_{i\in\mathcal{V}}B_\phi(x_i^\star, x_i^{t}).
\label{eqn: lemma Lyapunov positive}
\end{equation}
\label{lemma Lyapunov positive}
\end{lemma} 
The sketch of the proof is as follows. Use equation \eqref{KKT: dual} and \eqref{dual update} we can show \begin{equation*}
    \textstyle H(x^t, \mu^t)\geq -\frac{\omega}{2\tau}\norm{\mu^{t-1}-\mu^\star}_2^2-\frac{1}{2\omega\tau}\norm{\mu^{t}-\mu^{t-1}}_2^2.
\end{equation*}
In addition, equation \eqref{dual update} and Assumption~\ref{basic assumption}, particularly assumptions on function \(\phi\) and matrix \(P\), ensures that 
\[\textstyle-\frac{1}{2\omega\tau}\norm{\mu^t-\mu^{t-1}}_2^2+\frac{\tau}{2\omega\sigma}\sum_{i\in\mathcal{V}}B_{\phi}(x_i^\star,x_i^t)\geq 0.\]
Substitute these two inequalities into \eqref{Lyapunov function}, use \eqref{phi strong convexity} we can show \(V(t)\geq (\rho-\frac{\tau}{\omega\alpha\sigma})\sum_{i\in\mathcal{V}}B_\phi(x_i^\star, x_i^t)\), which, due to the assumption in \eqref{parameter}, finally reduces to \eqref{eqn: lemma Lyapunov positive}. Then positive definiteness of \(V(t)\) follows from the positive definiteness of Bregman divergence and the fact \(\frac{(1-\omega)\omega\alpha\sigma\rho+\gamma\rho}{(2-\omega)\omega\alpha\sigma}>0\) when \(0<\omega<1\).

Notice that \(V(t)\) is a random variable whose value depends on the realization of \(\mathcal{S}_{1:t}\), which is the history of selected node sets, \ie, \(\{\mathcal{S}_1, \mathcal{S}_2, \ldots, \mathcal{S}_{t}\}\). The following theorem shows that the expected value of \(V(t)\) conditioned on \(\mathcal{S}_{1:t}\), \ie , \(\mathds{E}_{\mathcal{S}_{1:t}}[V(t)]\) is monotonically non-increasing with respect to \(t\).  

\begin{theorem}[Global convergence]
Suppose that Assumption~\ref{basic assumption} . Let the sequence \(\{y^{t}, x^{t}, \mu^{t}\}\) be generated by Algorithm~\ref{alg: sBPDMM}. Let \(R(t+1)\) and \(V(t)\) be defined as in \eqref{residuals} and \eqref{Lyapunov function}, respectively. If \(\rho, \tau, \gamma, \omega\) satisfy
\eqref{parameter}, then we have the following monotonicity relation
\begin{equation*}
\mathds{E}_{\mathcal{S}_{1:t}}\left[V(t)\right]-\mathds{E}_{\mathcal{S}_{1:t+1}}\left[V(t+1)\right]\geq \mathds{E}_{\mathcal{S}_{1:t+1}}\left[R(t+1)\right].
\end{equation*}
\label{theorem global convergence}
\end{theorem}

The sketch of the proof is as follows. We substitute the subgradient in \eqref{primal updata: optimality} into \eqref{definition of subgradient} and obtain an inequality. Use three point property \eqref{3-point property} we can split the right hand side of this inequality into three parts, each contributes to \(R(t+1), V(t)\) and \(V(t+1)\), respectively. Taking the expectation over realization of \(\mathcal{S}_{t+1}\) conditioned on the value of \(x^t\), we obtain the following relation 
\begin{equation}
    \mathds{E}_{\mathcal{S}_{t+1}}[R(t+1)]\leq V(t)-\mathds{E}_{\mathcal{S}_{t+1}}[V(t+1)],
    \label{eqn: one step expectation}
\end{equation}
where assumptions in Assumption~\ref{basic assumption} and \eqref{parameter} ensures that all intermediate terms cancel each other. Taking the expectation over the realization of \(\mathcal{S}_{1:t}\) on both sides of \eqref{eqn: one step expectation}, we reach the inequality in Theorem~\ref{theorem global convergence}. 

Summing the inequality in Theorem~\ref{theorem global convergence} from the case of \(t=0\) to \(t={T-1}\) we have
\begin{equation}
    \textstyle \sum_{t=1}^{T} \mathds{E}_{\mathcal{S}_{1:t}}[R(t)]\leq V(0).
    \label{eqn: telescope series}
\end{equation}
Since \(\mathds{E}_{\mathcal{S}_{1:t}}[R(t)]\geq 0\) for all \(t\), inequality \eqref{eqn: telescope series} implies that \(\mathds{E}_{\mathcal{S}_{1:t}}[R(t)]\to 0\) as \(T\to\infty\), which establishes the global convergence of Algorithm~\ref{alg: sBPDMM}. In addition, if we apply Jensen's inequality to \eqref{eqn: telescope series}, we obtain the following corollary, which shows the the \(O(1/T)\) iteration complexity of Algorithm~\ref{alg: sBPDMM} in an ergodic sense.

\begin{corollary}[Iteration complexity]
Suppose that Assumption~\ref{basic assumption} holds. Let the sequence \(\{y^{t}, x^{t}, \mu^{t}\}\) be generated by Algorithm~\ref{alg: sBPDMM}. Let \(V(t)\) be defined as in \eqref{Lyapunov function}, \(\overline{x}^{T}=\frac{1}{T}\sum_{t=0}^{T-1}x^{t}\). If \(\rho, \tau, \gamma, \omega\) satisfy
\eqref{parameter}, then
\begin{align*}
&\mathds{E}_{\mathcal{S}_{1:T}}\left[L(\overline{x}^T, \mu^\star)-L(x^\star, \mu^\star)\right]\leq  \frac{V(0)}{\omega T}\\
&\mathds{E}_{\mathcal{S}_{1:T}}\left[\frac{1}{2}\norm{((I_{|\mathcal{V}|}-P)\otimes I_n)\overline{x}^{T}}_2^2\right]\leq  \frac{V(0)}{\gamma \rho T}
\end{align*}
\label{corollary 1/T convergence}
\end{corollary}

The bound on running duality gap was used in \cite{meng2015proximal}.

%% file: experiments.tex
\section{Numerical examples}\label{section: experiments}
In this section, we demonstrate the effectiveness and efficiency of Algorithm~\ref{alg: sBPDMM} via numerical examples. 

Consider the an instance of problem \eqref{opt: dist opt} where \(f_i(x_i)=\langle c_i, x_i\rangle\) and \(\mathcal{X}=\{u\in\mathbb{R}^{n}_+|\norm{u}_1=1\}\) is the probability simplex, \(\mathcal{G}=(\mathcal{V}, \mathcal{E})\) is a undirected connected communication graph. Such optimizaton can model, for example, multi-agent decision making, where \(c_i\) is the cost of agent \(i\) for choosing policy \(x_i\). 

We generate an instance of this optimization where entries of \(c_1, \ldots, c_{|\mathcal{V}|}\in\mathbb{R}^{100}\) are sampled from standard normal distribution. \(\mathcal{G}\) is a randomly generated with \(|\mathcal{V}|=100\) and edge probability \(0.2\) \cite[p.~90]{mesbahi2010graph}. Matrix \(P\) is obtained by minimizing its second largest eigenvalue (in this case, \(\lambda_2(P)=0.4786\)) while preserving graph adjacency constraints. We choose the following parameters in Algorithm~\ref{alg: sBPDMM}:
\begin{itemize}
    \item \(\phi(u)=\sum_{k=1}^nu[k]\ln u[k]\), where \(u[k]\) denotes the \(k\)-th element of vector \(u\). Then assumption in \eqref{phi strong convexity} is satisfied by \(\alpha=1, p=1\) (see Remark~1 in \cite{wang2014bregman}).
    \item \(\rho=1, \tau=\omega/(4-2\omega)\). Notice that assumptions in \eqref{parameter} are satisfied with \(\gamma=\omega/2\).
\end{itemize}

With these assumptions, the mirror averaging step \eqref{mirror averaging} and local optimization step \eqref{primal update} reduces to the following (see Section 4.3 in \cite{bubeck2015convex} for details)
\begin{subequations}
    \begin{align}
         y_i^t=&\text{Proj}\left[\textstyle \prod_{j\in\mathcal{N}(i)}(x_j^t)^{P_{ij}}\right]\\
         x_i^t=&\text{Proj}\left[\textstyle y_i^t\exp\frac{-c_i-\mu_i+\sum_{j\in\mathcal{N}(i)}P_{ij}\mu_j}{\rho}\right]
    \end{align}
    \label{eqn: sBPDMM implementation}
\end{subequations}
where multiplication, power and exponential operation on vectors are all elementwise, and \(\text{Proj}[u]=u/\norm{u}_1\) for all \(u\in\mathbb{R}^n\). Update  \eqref{eqn: sBPDMM implementation} amounts to elementwise operation that allows massive parallel implementation.

We demonstrate the convergence performance of Algorithm~\ref{alg: sBPDMM} in Fig.~\ref{fig: omega} and Fig.~\ref{fig: phi}, where \(f^t\) and \(f^\star\) are the objective function value achieved at iteration \(t\) and, respectively, optimality. In particular, Fig.~\ref{fig: omega} shows that as \(\omega\) increases, the convergence of Algorithm~\ref{alg: sBPDMM} becomes faster and less oscillating, which is because more nodes get updated at each iteration. Fig.~\ref{fig: phi} shows that when we choose \(\phi\) as negative entropy function rather than quadratic function, the convergence speed is improved dramatically. This is because compared with quadratic function, negative entropy function exploits the structure of probability simplex much better. Such improvement demonstrates the advantage of Algorithm~\ref{alg: sBPDMM} over stochastic multiplier methods based on quadratic augmentation \cite{wei20131, wang2014parallel,zhu2016stochastic}.

\begin{figure}[ht]
    \centering
    \includegraphics[width=0.9\linewidth]{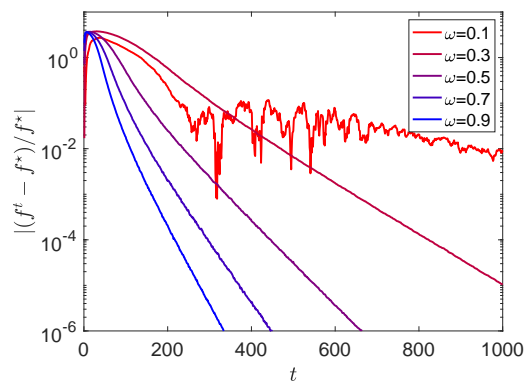}
    \caption{Comparison of different \(\omega\) values}
    \label{fig: omega}
\end{figure}
\begin{figure}[ht]
    \centering
    \includegraphics[width=0.9\linewidth]{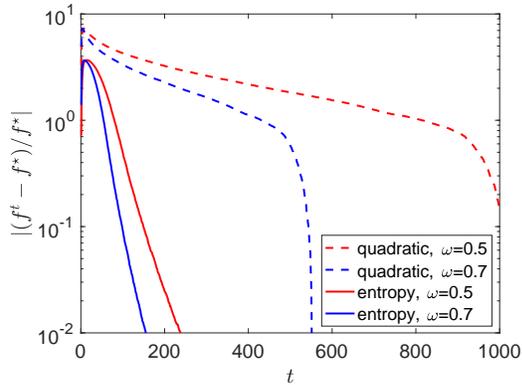}
    \caption{Comparison of different \(\phi\) function}
    \label{fig: phi}
\end{figure}

%% file: conclusion.tex
\section{Conclusions}\label{section: conclusion}
In this paper, we generalize BPDMM \cite{yu2018bregman} to stochastic BPDMM, where each iteration only solves local optimization on a randomly selected subset of nodes rather than all the nodes in the network. Such generalization requires less number of processors running in parallel, hence allows application to much larger scale networks. Future directions include generalization to directed and time varying networks.   

%% file: appendix.tex
For notation simplicity, we let \(Q\coloneqq I_{|\mathcal{V}|}-P\). Suppose Assumption~\ref{basic assumption} holds, then the nullspace of \(I_{|\mathcal{V}|}-P\) is spanned by \(\mathbf{1}_{|\mathcal{V}|}\)
In addition, Assumption \eqref{basic assumption} and update rule \eqref{sBPDMM primal update} ensure that
\begin{subequations}
\begin{align}
(Q\otimes I_n)x^\star=&0\label{eqn: consensus assumption}\\
\delta_{\mathcal{X}}(x_i^t)=\delta_{\mathcal{X}}(x_i^\star)=&0, \enskip \forall i\in\mathcal{V}\label{eqn: feasibility assumption}
\end{align}
\end{subequations} 
for all \(t\).
We will need the following lemmas.
\begin{lemma} Let 
\begin{equation}
    y_i^t=\underset{y_i\in\mathcal{X}}{\argmin}\sum_{j\in\mathcal{N}(i)}P_{ij}B_\phi(y_i, x_j^t),
    \label{eqn: mirror averaging appendix}
\end{equation}
for all \(i\in\mathcal{V}\). Then for any \(u\in\mathcal{X}\), 
\begin{equation}
\sum_{i\in\mathcal{V}} \left(B_\phi(u, x_i^{t})- B_\phi(u, y_i^{t})\right)\geq\sum_{i,j\in\mathcal{V}}P_{ij}B_\phi(y_i^{t}, x_j^{t})
\label{lemma Pythagorean: eqn1}
\end{equation}\label{lemma Pythagorean}
\end{lemma}
\begin{proof}
Equation \eqref{eqn: mirror averaging appendix} holds if and only if: for any \(u\in\mathcal{X}\), 
\begin{equation*}
\sum_{j\in\mathcal{V}}P_{ij}\langle \nabla\phi(y_i^{(t)})-\nabla \phi(x_j^{(t)}), u-y_i^{(t)}\rangle\geq 0
\end{equation*}
Using three point property \eqref{3-point property}, we have
\begin{equation}
\begin{aligned}
&\sum_{j\in\mathcal{V}}P_{ij}B_\phi(u, x_j^{(t)})-\sum_{j\in\mathcal{V}}P_{ij}B_\phi(u, y_i^{(t)}) \\
\geq&\sum_{j\in\mathcal{V}}P_{ij}B_\phi(y_i^{(t)}, x_j^{(t)})
\end{aligned}
\label{lemma Pythagoream: eqn2}
\end{equation}
Summing \eqref{lemma Pythagoream: eqn2} over all \(i\in\mathcal{V}\) completes the proof.
\end{proof}

\begin{lemma}
Suppose Assumption~\ref{basic assumption} holds. Then
\begin{equation}
\begin{aligned}
&\sigma\norm{(Q\otimes I_n)u}_2^2\leq \sum_{i,j\in\mathcal{V}}P_{ij} \norm{u_i-v_j}_p^2
\label{lemma residual & variance: eqn1}
\end{aligned}
\end{equation}
for all \(u, v\in\mathcal{X}^{|\mathcal{V}|}\), where \(\norm{\cdot}_p\) denote \(l_p\) norm and \(\sigma=\min\{1, n^{\frac{2}{p}-1}\}\).
\label{lemma residual & variance}
\end{lemma}
\begin{proof}
First, observe that if \(P\) is symmetric, stochastic, irreducible and positive semi-definite,  \(P-P^2\) is positive semi-definite \cite[Theorem 8.4.4]{horn2012matrix}. Since \(P\mathbf{1}_{|\mathcal{V}|}=P^\top\mathbf{1}_{|\mathcal{V}|}=\mathbf{1}_{|\mathcal{V}|}\), we can show the following
\begin{equation*}
\begin{aligned}
&\sum_{i, j\in\mathcal{V}}P_{ij}\norm{\sum_{k\in\mathcal{V}} P_{ik}u_k-u_j}_2^2\\
= & \norm{u}_2^2-\norm{(P\otimes I_n) u}_2^2\\
\geq &\norm{u}_2^2-\norm{(P\otimes I_n)u}_2^2 - 2\langle u, ((P-P^2)\otimes I_n)u\rangle\\
=&\norm{(Q\otimes I_n)u}_2^2
\end{aligned}
\end{equation*} 
Hence \eqref{lemma residual & variance: eqn1} holds due to the fact that 
\begin{equation*}
\sum_{k\in\mathcal{V}} P_{ik}u_k = \underset{w\in\mathcal{X}}{\argmin}\sum_{j\in\mathcal{V}} P_{ij}\norm{w-u_j}_2^2,
\end{equation*}
for all \(i\in\mathcal{V}\), and that \(\norm{w}_2^2\leq 1/\sigma\norm{w}_p^2\) for all \(w\in\mathbb{R}^n\) where \(\sigma = \min\{1, n^{\frac{2}{p}-1}\}\).
\end{proof}

\subsection{Lemma~\ref{lemma Lyapunov positive}}
\begin{proof}
Using \eqref{eqn: consensus assumption} and \eqref{eqn: Lagrangian}  we can show that
\begin{equation}
\begin{aligned}
&L(x^t, \mu^t)-L(x^\star, \mu^\star)\\
= &\sum_{i\in\mathcal{V}}\left((f_i+\delta_\mathcal{X})(x_i^{t})- (f_i+\delta_\mathcal{X})(x_i^\star)\right)\\
&+\langle\mu^t, (Q\otimes I_n)x^t\rangle\overset{\eqref{KKT: dual} }{\geq} \langle \mu^t-\mu^\star, (Q\otimes I_n)x^{t}\rangle 
\end{aligned}
\label{lemma Lyapunov positive: eqn1}
\end{equation}
Substitute \eqref{lemma Lyapunov positive: eqn1} into \eqref{generalized Lagrangian} we have
\begin{equation}
\begin{aligned}
&H(x^t, \mu^t)\\
 \geq &\langle \mu^{t}-\mu^\star, (Q\otimes I_n)x^{t}\rangle-\tau\norm{(Q\otimes I_n)x^{t}}_2^2\\
 \overset{\eqref{dual update} }{=} &\frac{1}{\tau}\langle \mu^{t-1}-\mu^\star, \mu^t-\mu^{t-1}\rangle\\
  \geq & -\frac{\omega}{2\tau}\norm{\mu^{t-1}-\mu^\star}_2^2-\frac{1}{2\omega\tau}\norm{ \mu^t-\mu^{t-1}}_2^2
\end{aligned}
\label{lemma Lyapunov positive: eqn2}
\end{equation}
where the last step is due to \(2\langle a, b\rangle\geq -\norm{a}_2^2-\norm{b}_2^2\). Therefore, substitute \eqref{lemma Lyapunov positive: eqn2} into \eqref{Lyapunov function} we have 

\begin{equation}
\begin{aligned}
 V(t)\geq  & \rho\sum_{i\in\mathcal{V}} B_\phi(x_i^\star, x_i^{t})-\frac{1}{2\omega\tau}\norm{ \mu^t-\mu^{t-1}}_2^2 \\
 \overset{\eqref{phi strong convexity}}{\geq }&
 \left(\rho-\frac{\tau}{\omega\alpha\sigma}\right)\sum_{i\in\mathcal{V}}B_\phi(x_i^\star, x_i^{t})\\
 &+\frac{\tau}{2\omega\sigma} \sum_{i\in\mathcal{V}} \norm{x_i^{t}-x_i^\star}_p^2-\frac{1}{2\omega\tau}\norm{\mu^t-\mu^{t-1}}_2^2 \\
  \overset{\eqref{parameter} }{\geq} &\frac{(1-\omega)\omega\alpha\sigma\rho+\gamma\rho}{(2-\omega)\omega\alpha\sigma}\sum_{i\in\mathcal{V}}B_\phi(x_i^\star, x_i^{t}) \\
  & +\frac{\tau}{2\omega\sigma}\left(\sum_{i\in\mathcal{V}} \norm{x_i^{t}-x_i^\star}_p^2-\frac{\sigma}{\tau^2}\norm{\mu^t-\mu^{t-1}}_2^2\right) 
\label{eqn: Lyapunov positivity}
\end{aligned}
\end{equation}
Since \(x^\star_i=x_j^\star\) for all \(i,j\in\mathcal{V}\), we have 
\begin{equation*}
\begin{aligned}
0 \overset{\eqref{lemma residual & variance: eqn1} }{\leq} & \sum_{i, j\in\mathcal{V}}P_{ij} \norm{x_i^{t}-x_j^\star}_p^2-\sigma\norm{ (Q\otimes I_n)x^{t}}_2^2  \\
    =&\sum_{i, j\in\mathcal{V}}P_{ij} \norm{x_i^{t}-x_i^\star}_p^2-\sigma\norm{ (Q\otimes I_n)x^{t}}_2^2 \\
    \overset{\eqref{dual update} }{=}&\sum_{i\in\mathcal{V}} \norm{x_i^{t}-x_i^\star}_p^2-\frac{\sigma}{\tau^2}\norm{\mu^t-\mu^{t-1}}_2^2
    \end{aligned}
\end{equation*}
Substitute the above inequality into \eqref{eqn: Lyapunov positivity} we obtain \eqref{eqn: lemma Lyapunov positive}.
\end{proof}
\subsection{Theorem~\ref{theorem global convergence} }
\begin{proof}
Let \(q_i\) be the \(i\)-th column of \(Q\). Since \(f+\delta_\mathcal{X}\) is convex, the subgradient in \eqref{primal updata: optimality} satisfy the following
\begin{equation}
\begin{aligned}
&\sum_{i\in\mathcal{S}_{t+1}} f_{i}(x_{i}^{t+1})-\sum_{i\in\mathcal{S}_{t+1}} f_{i}(x_i^\star)\\
\leq & \sum_{{i}\in\mathcal{S}_{t+1}}\langle -\mu^{t}, (q_{i}\otimes I_n)(x_{i}^{t+1} -x_i^\star) \rangle\\
& + \rho\sum_{i\in\mathcal{S}_{t+1}}\langle\nabla\phi(x_{i}^{t+1})-\nabla\phi(y_{i}^{t}), x_i^\star-x_{i}^{t+1} \rangle,
\end{aligned}
\label{theorem global convergence: eqn1}
\end{equation}
where we use \eqref{eqn: feasibility assumption}.

The first term on the RHS of \eqref{theorem global convergence: eqn1} can be rewritten as 
\begin{equation}
\begin{aligned}
& \sum_{{i}\in\mathcal{S}_{t+1}}\langle -\mu^{t}, (q_{i}\otimes I_n)(x_{i}^{t+1} -x_i^\star) \rangle\\
 \overset{\eqref{not updated variable}}{=} & \sum_{{i}\in\mathcal{S}_{t+1}}\langle -\mu^{t}, (q_{i}\otimes I_n)(x_{i}^{t} -x_i^\star)  \rangle\\
 & + \langle \mu^{t}, (Q\otimes I_n)x^{t}  \rangle- \langle \mu^{t}, (Q\otimes I_n)x^{t+1} \rangle\\
 \overset{\eqref{dual update} }{=} &-\sum_{{i}\in\mathcal{S}_{t+1}}\langle \mu^{t}, (q_{i}\otimes I_n)(x_{i}^{t} -x_i^\star)  \rangle + \langle \mu^{t}, (Q\otimes I_n)x^{t}  \rangle\\
 & - \langle \mu^{t+1}, (Q\otimes I_n)x^{t+1}  \rangle + \tau\norm{(Q\otimes I_n)x^{t+1}}_2^2
\end{aligned}
\label{theorem global convergence: eqn2}
\end{equation} 
To simplify the second term on the RHS of \eqref{theorem global convergence: eqn1}, notice that
\begin{equation}
\begin{aligned}
& \sum_{i\in\mathcal{S}_{t+1}}\langle\nabla\phi(x_{i}^{t+1})-\nabla\phi(y_{i}^{t}), x_i^\star-x_{i}^{t+1} \rangle\\
&\overset{\eqref{3-point property} }{=}  \sum_{i\in\mathcal{S}_{t+1}} \left(B_\phi(x_i^\star, y_{i}^{t}) -  B_\phi(x_i^\star, x_{i}^{t+1})-B_\phi(x_{i}^{t+1}, y_{i}^{t})\right)\\
&\overset{\eqref{not updated variable} }{=} 
 \sum_{i\in\mathcal{S}_{t+1}} \left(B_\phi(x_i^\star, y_{i}^{t}) -  B_\phi(x_i^\star, x_{i}^{t})\right) +\sum_{i\in\mathcal{V}}  B_\phi(x_i^\star, x_{i}^{t})\\
 &-\sum_{i\in\mathcal{V}}  B_\phi(x_i^\star, x_{i}^{t+1})  - \sum_{i\in\mathcal{S}_{t+1}}  B_\phi(x_{i}^{t+1}, y_{i}^{t})
\end{aligned}
\label{theorem global convergence: eqn3}
\end{equation}

Substitute \eqref{theorem global convergence: eqn2} and \eqref{theorem global convergence: eqn3} into \eqref{theorem global convergence: eqn1}, we have 
\begin{equation}
\begin{aligned}
&\sum_{i\in\mathcal{S}_{t+1}} f_{i}(x_{i}^{t+1})-\sum_{i\in\mathcal{S}_{t+1}} f_{i}(x_i^\star)\\
\leq &-\sum_{{i}\in\mathcal{S}_{t+1}}\langle \mu^{t}, (q_{i}\otimes I_n)(x_{i}^{t} -x_i^\star)  \rangle  + \langle \mu^{t}, (Q\otimes I_n)x^{t}  \rangle\\
 & - \langle \mu^{t+1}, (Q\otimes I_n)x^{t+1}  \rangle + \tau\norm{(Q\otimes I_n)x^{t+1}}_2^2\\
 &+\rho \sum_{i\in\mathcal{S}_{t+1}} \left(B_\phi(x_i^\star, y_{i}^{t}) -  B_\phi(x_i^\star, x_{i}^{t})\right)\\
&+\rho\sum_{i\in\mathcal{V}}  B_\phi(x_i^\star, x_{i}^{t})-\rho\sum_{i\in\mathcal{V}}  B_\phi(x_i^\star, x_{i}^{t+1})\\
&- \rho\sum_{i\in\mathcal{S}_{t+1}}  B_\phi(x_{i}^{t+1}, y_{i}^{t})
\end{aligned}
\label{theorem global convergence: eqn5}
\end{equation}
In addition, notice that 
\begin{equation}
\begin{aligned}
& \sum_{{i}\in\mathcal{S}_{t+1}} \left( f_{i}(x_{i}^{t})- f_{i}(x_i^\star)\right)\overset{\eqref{not updated variable}}{=}\\
&\sum_{{i}\in\mathcal{S}_{t+1}} \left(f_{i}(x_{i}^{t+1})
- f_{i}(x_i^\star)\right) +\sum_{i\in\mathcal{V}} \left(f_i(x_i^{t})- f_i(x_i^{t+1})\right)
\end{aligned}
\label{theorem global convergence: eqn6}
\end{equation}
Substitute \eqref{theorem global convergence: eqn5} into \eqref{theorem global convergence: eqn6}, we have 
\begin{equation}
\begin{aligned}
& \sum_{{i}\in\mathcal{S}_{t+1}} \left(f_{i}(x_{i}^{t})- f_{i}(x_i^\star)\right)\\
\leq  & H(x^{t}, \mu^{t})-H(x^{t+1}, \mu^{t+1}) + \tau\norm{(Q\otimes I_n)x^{t}}_2^2\\
& - \sum_{{i}\in\mathcal{S}_{t+1}}\langle \mu^{t}, (q_{i}\otimes I_n)(x_{i}^{t}-x_i^\star)  \rangle\\
&+\rho \sum_{i\in\mathcal{S}_{t+1}} \left(B_\phi(x_i^\star, y_{i}^{t}) -  B_\phi(x_i^\star, x_{i}^{t})\right)\\
&+\rho\sum_{i\in\mathcal{V}} B_\phi(x_i^\star, x_{i}^{t})-\rho\sum_{i\in\mathcal{V}} B_\phi(x_i^\star, x_{i}^{t+1}) \\
&  - \rho\sum_{i\in\mathcal{S}_{t+1}}  B_\phi(x_{i}^{t+1}, y_{i}^{t})\\
\end{aligned}
\label{theorem global convergence: eqn7}
\end{equation}
where we use the definition in  \eqref{residuals}.

Taking the expectation of \eqref{theorem global convergence: eqn7} over \(\mathcal{S}_{t+1}\) conditioned on \(x^{t}\), we have the following
\begin{equation}
\begin{aligned}
& \omega\sum_{{i}\in\mathcal{V}} \left(f_{i}(x_{i}^{t})- f_{i}(x_i^\star)\right)\\
\leq & H(x^{t}, \mu^{t})-\mathds{E}_{\mathcal{S}_{t+1}}\left[H(x^{t+1}, \mu^{t+1})\right]\\
&+ \tau\norm{(Q\otimes I_n)x^{t}}_2^2 - \omega\langle \mu^{t}, (Q\otimes I_n)x^{t}  \rangle  \\
& + \rho\omega\sum_{{i}\in\mathcal{V}}\left( B_\phi(x_i^\star, y_{i}^{t}) -  B_\phi(x_i^\star, x_{i}^{t})\right)\\
&+  \rho\sum_{i\in\mathcal{V}} B_\phi(x_i^\star, x_i^{t})- \rho\mathds{E}_{\mathcal{S}_{t+1}}\left[ \sum_{i\in\mathcal{V}}  B_\phi(x_i^\star, x_i^{t+1}) \right] \\
& - \rho\mathds{E}_{\mathcal{S}_{t+1}}\left[\sum_{i\in\mathcal{S}_{t+1}} B_\phi(x_{i}^{t+1}, y_{i}^{t}) \right]\\
\end{aligned}
\label{theorem global convergence: eqn8}
\end{equation}
where we use \eqref{eqn: consensus assumption}.
Here we assume \(y^t_i\) is computed as in \eqref{BPDMM: mirror averaging} for all nodes in \(\mathcal{V}\), even though Algorithm~\ref{alg: BPDMM} only require computation on nodes in \(\mathcal{S}_{t+1}\).
Substitute \eqref{eqn: feasibility assumption} into \eqref{eqn: Lagrangian} we have
\begin{equation}
    \begin{aligned}
    &\sum_{i\in\mathcal{V}} (f_i(x_i^t)-f_i(x_i^\star))\\
    = &L(x^t, \mu^\star)-L(x^\star, \mu^\star)-\langle\mu^\star, (Q\otimes I_n)x^t \rangle
    \end{aligned}
    \label{theorem global convergence: eqn9}
\end{equation}

Combine \eqref{theorem global convergence: eqn8} and \eqref{theorem global convergence: eqn9} we have
\begin{equation}
\begin{aligned}
&\mathds{E}_{\mathcal{S}_{t+1}}\left[R(t+1)\right]\\
\leq & H(x^{t}, \mu^{t})-\mathds{E}_{\mathcal{S}_{t+1}}\left[H(x^{t+1}, \mu^{t+1})\right]\\
& - \omega\langle \mu^{t}-\mu^\star, (Q\otimes I_n)x^{t}  \rangle  \\
& + \rho\omega\sum_{{i}\in\mathcal{V}}\left( B_\phi(x_i^\star, y_{i}^{t}) -  B_\phi(x_i^\star, x_{i}^{t})\right)\\
&+  \rho\sum_{i\in\mathcal{V}}  B_\phi(x_i^\star, x_i^{t})- \rho\mathds{E}_{\mathcal{S}_{t+1}}\left[\sum_{i\in\mathcal{V}}  B_\phi(x_i^\star, x_i^{t+1})\right] \\
&  + (\tau+\frac{\rho\gamma}{2})\norm{(Q\otimes I_n)x^{t}}_2^2
\end{aligned}
\label{theorem global convergence: eqn10}
\end{equation}

Using \eqref{3-point property} and \eqref{dual update} we can show
\begin{equation}
\begin{aligned}
& -\langle \mu^{t}-\mu^\star, (Q\otimes I_n)x^{t}\rangle = \frac{1}{2\tau}\norm{\mu^\star-\mu^{t-1}}_2^2\\
&-\frac{1}{2\tau}\norm{\mu^\star-\mu^{t}}_2^2-\frac{\tau}{2}\norm{(Q\otimes I_n)x^{t}}_2^2
\end{aligned}
\label{theorem global convergence: eqn11}
\end{equation}
Substitue \eqref{theorem global convergence: eqn11} into \eqref{theorem global convergence: eqn10}, use the definition in \eqref{Lyapunov function} we have
\begin{equation}
\begin{aligned}
&\mathds{E}_{\mathcal{S}_{t+1}}\left[R(t+1)\right] \\
\leq &V(t)-\mathds{E}_{\mathcal{S}_{t+1}}\left[V(t+1)\right]\\
& + \rho\omega\sum_{{i}\in\mathcal{V}} \left(B_\phi(x_i^\star, y_{i}^{t}) - B_\phi(x_i^\star, x_{i}^{t})\right)\\
& +\left(\tau+\frac{\rho\gamma}{2}-\frac{\omega\tau}{2}\right)\norm{(Q\otimes I_n)x^{t}}_2^2
\end{aligned}
\label{theorem global convergence: eqn12}
\end{equation}
Since 
\begin{equation}
    \begin{aligned}
    &\sum_{{i}\in\mathcal{V}} (B_\phi(x_i^\star, y_{i}^{t}) -  B_\phi(x_i^\star, x_{i}^{t}))\\ \overset{\eqref{lemma Pythagorean: eqn1}}{\leq}& -\sum_{i, j\in\mathcal{V}} P_{ij}B_\phi(y_i^t, x_j^t)
    \overset{\eqref{phi strong convexity}}{\leq} -\frac{\alpha}{2}\sum_{{i, j}\in\mathcal{V}}P_{ij}\norm{y_i^t-x_i^t}_p^2\\
    \overset{\eqref{lemma residual & variance: eqn1} }{\leq}& -\frac{\alpha\sigma}{2} \norm{(Q\otimes I_n) x^t}_2^2
    \end{aligned}
    \label{theorem global convergence: eqn13}
\end{equation}
Substitute \eqref{theorem global convergence: eqn11} into \eqref{theorem global convergence: eqn10} we have
\begin{equation}
\begin{aligned}
&\mathds{E}_{\mathcal{S}_{t+1}}\left[R(t+1)\right]\\
 \leq &V(t)-\mathds{E}_{\mathcal{S}_{t+1}}\left[V(t+1)\right]\\
 &+ \frac{(2-\omega)\tau+\rho(\gamma-\omega\alpha\sigma)}{2}\norm{(Q\otimes I_n)x^{t}}_2^2\\
 \overset{\eqref{parameter} }{\leq} & V(t)-\mathds{E}_{\mathcal{S}_{t+1}}\left[V(t+1)\right].
\end{aligned}
\label{theorem global convergence: eqn14}
\end{equation}

Taking the expectation of \eqref{theorem global convergence: eqn10} over realization of \(\mathcal{S}_{1:t}\) we obtain the desired results.
\end{proof}